\documentclass[11pt]{amsart}
\usepackage{amssymb}

\usepackage{color}

\newenvironment{theorem}{\subsection{}{\textbf {Theorem.}}\em}{}
\newenvironment{proposition}{\subsection{}{\textbf {Proposition.}}\em}{}
\newenvironment{corollary}{\subsection{}{\textbf {Corollary.}}\em}{}

\newenvironment{definition}{\subsection{}{\textbf {Definition.}}\em}{\smallskip}

\newenvironment{remark}{\subsection{}{\textbf {Remark.}}}{\smallskip}
\newenvironment{remarks}{\subsection{}{\textbf {Remarks.}}}{\smallskip}

\newcommand{\norm}[1]{\lVert#1\rVert}

\renewcommand{\le}{\leqslant}
\renewcommand{\ge}{\geqslant}
\renewcommand{\mid}{\::\:}

\def\bofh{\mathcal B\left(\mathcal H\right)}

\def\bbC{\mathbb C}

\def\bbM{\mathbb M}

\def\cA{\mathcal A}
\def\cB{\mathcal B}
\def\cC{\mathcal C}
\def\cD{\mathcal D}

\def\cH{\mathcal H}

\def\cK{\mathcal K}

\def\cM{\mathcal M}
\def\cN{\mathcal N}

\def\cS{\mathcal S}

\def\cZ{\mathcal Z}

\DeclareMathOperator{\Lat}{Lat}

\def\dotplus{\overset{_{_{\,\,_\bullet}}}{+}}
\begin{document}

\title[Abelian, amenable operator algebras are similar to $C^*$-algebras]{Abelian, amenable operator algebras are similar to $C^*$-algebras}


%
%===================================================================
%

\begin{abstract}
Suppose that $H$ is a complex Hilbert space and that $\cB(H)$ denotes the bounded linear operators on $H$.  We show that every abelian, amenable operator algebra is similar to a $C^*$-algebra.   We do this by showing that if $\cA \subseteq \cB(H)$ is an abelian algebra with the property that given any bounded representation $\varrho: \cA \to \cB(H_\varrho)$ of $\cA$ on a Hilbert space $H_\varrho$, every invariant subspace of $\varrho(\cA)$ is topologically complemented by another invariant subspace of $\varrho(\cA)$, then $\cA$ is similar to an abelian $C^*$-algebra.   
\end{abstract}

%
%===================================================================
%

\author[L.W. Marcoux]{Laurent W.~Marcoux${}^1$}
\email{LWMarcoux@uwaterloo.ca}
\address
	{Department of Pure Mathematics\\
	University of Waterloo\\
	Waterloo, Ontario \\
	Canada  \ \ \ N2L 3G1}
\author[A.~I.~Popov]{Alexey I. Popov}
\email{a4popov@uwaterloo.ca}
\address
	{Department of Pure Mathematics\\
	University of Waterloo\\
	Waterloo, Ontario \\
	Canada  \ \ \ N2L 3G1}
\thanks{${}^1$ Research supported in part by NSERC (Canada)}
\date\today

\maketitle

%
%===================================================================
%
%===================================================================
%
%===================================================================
%

\section{Introduction.}

\subsection{}
Let $\cA$ be a Banach algebra and $X$ be a Banach space which is also a bimodule over $\cA$.   We say that $X$ is a \textbf{Banach bimodule} over $\cA$ if the module operations are continuous; that is, if there exists $\kappa > 0$ so that $\norm {a x} \le \kappa \norm {a} \ \norm {x}$, and $\norm {x b} \le \kappa \norm {x}\ \norm {b}$ for all $a, b \in \cA$ and $x \in X$.   

Given a Banach bimodule $X$ over $\cA$, we introduce an action of $\cA$ upon the dual space $X^*$ of $X$ under which $X^*$ becomes a \textbf{dual Banach $\cA$-bimodule}.   This is the so-called \textbf{dual action}:
\[
(a  x^*)(x) = x^*(x a) \mbox{ \ \ \ \ \ and \ \ \ \ \ } (x^*  a)(x) = x^*(a x) \]
for all $a \in \cA$, $x \in X$, $x^* \in X^*$.

A (continuous) \textbf{derivation} from a Banach algebra $\cA$ into a Banach $\cA$-bimodule $X$ is a continuous linear map $\delta: \cA \to X$ satisfying $\delta (a b) = a  \delta(b) + \delta(a)  b$ for all $a, b \in \cA$.   For any fixed $z \in X$, the map $\delta_z : \cA \to X$ defined by $\delta_z (a) = a  z - z  a$ is a derivation with $\norm {\delta_z} \le 2 \norm {z}$.   Derivations of this type are said to be \textbf{inner}, and the algebra $\cA$ is said to be \textbf{amenable} if every continuous derivation of $\cA$ into a dual Banach bimodule $X$ is inner.

The notion of amenability of Banach algebras was introduced by B.~Johnson in his 1972 monograph~\cite{Joh1972}.   He showed that a locally compact topological group $G$ is amenable as a group  - that is, $G$ admits a left translation-invariant mean - if and only if the corresponding group algebra $(L^1(G), \norm {\cdot}_1)$ is amenable as a Banach algebra.  It is a standard and relatively straightforward exercise to show that if $\cA$ and $\cB$ are Banach algebras, $\varphi: \cA \to \cB$ is a continuous homomorphism with dense range, and if $\cA$ is amenable, then $\cB$ is amenable also.   

For $C^*$-algebras acting on a Hilbert space, the notion of amenability coincides with that of \emph{nuclearity}.   A $C^*$-algebra $\cB$ is said to be \textbf{nuclear} if there exists a directed set $\Lambda$ and two families  $\varphi_\lambda: \cB \to \mathbb{M}_{k(\lambda)}(\mathbb{C})$   and   $\psi_\lambda: \mathbb{M}_{k(\lambda)}(\mathbb{C}) \to \cB, \ \lambda \in \Lambda$ of completely positive contractions,  where $k(\lambda) \in \mathbb{N}$ for all $\lambda \in \Lambda$, so that 
\[
\lim_\lambda \norm {\psi_\lambda \circ \varphi_\lambda (b) - b} = 0 \mbox{   for all } b \in \cB. \]

It was shown by A.~Connes~\cite{Con1978}  that every amenable $C^*$-algebra is nuclear, while the converse - namely that every nuclear $C^*$-algebra is amenable - was established by U.~Haagerup~\cite{Haa1983}.

Let $H$ be a complex Hilbert space and denote by $\cB(H)$ the algebra of all bounded linear operators acting on $H$.    It follows from our observation above that if $\cD$ is a nuclear $C^*$-algebra and if $\varrho: \cD \to \cB(H)$ is a continuous representation of $\cD$, then $\overline{\varrho({\cD})}$ is an amenable algebra of operators in $\cB(H)$.     It is also known that any abelian $C^*$-algebra is nuclear (cf.~\cite{BO2008}, Proposition~2.4.2), as is  the algebra $\cK(H)$ of compact operators on $H$ (cf.~\cite{BO2008}, Proposition~2.4.1).    
In 1955, R.V.~Kadison raised the following question, now known as  \textbf{Kadison's Similarity Problem}~\cite{Kad1955}:
Let $\cD$ be a $C^*$-algebra, and suppose that $\varrho: \cD \to \cB(H_\varrho)$ is a continuous representation of $\cD$ on some Hilbert space $H_\varrho$.  For $S \in \cB(H)$ invertible, denote by $\mathrm{Ad}_S: \cB(H) \to \cB(H)$ the map $\mathrm{Ad}_S (X) = S^{-1} X S$.    Does there exist an invertible operator $S \in \cB(H_\varrho)$ so that $\tau:= \mathrm{Ad}_S \circ \varrho$ is a ${}^*$-homomorphism of $\cD$?	

\smallskip	

While the problem in this generality remains unsolved, it has been shown by E.~Christensen~\cite{Chr1981} to admit a positive answer whenever $\cD$ is irreducible (i.e. $\cD$ admits no invariant subspaces) and when $\cD$ is nuclear.   In particular, therefore, it holds when $\cA$ is abelian.     Haagerup~\cite{Haa1981} showed that if $\cD$ admits a \textbf{cyclic} vector, (i.e. there exists $x \in H$ so that $H = \overline{\cD x}$, then again, every continuous representation of $\cD$ is similar to a ${}^*$-representation.  

It follows from Christensen's work that  if a closed subalgebra $\cA \subseteq \cB(H)$ is a homomorphic image of an abelian $C^*$-algebra, then $\cA$ is necessarily amenable (and abelian), and that $\cA$ is similar to a $C^*$-algebra.

The converse problem is the following:

\begin{quote}
{\textbf{Question A.}} Is every amenable algebra of Hilbert space operators  a continuous, homomorphic image of (and hence similar to) a nuclear $C^*$-algebra?
\end{quote}

\smallskip
This problem has circulated since the 1980s.  It has been ascribed to Pisier, to Curtis and Loy, to \v Se\u\i nberg, and to Helemskii, amongst others.  For certain special classes of algebras, the question has been answered affirmatively.

Observe that if an amenable algebra $\cA \subseteq \cB(H)$ is similar to a $C^*$-algebra, then it must necessarily be semisimple.  In that regard, it is interesting to note that C.J.~Read~\cite{Rea2000} has constructed an example of an abelian, radical, amenable Banach algebra.      As a consequence of Corollary~\ref{cor3.2} below, the only continuous representation of Read's algebra on a Hilbert space is the trivial representation.  Thus ours is very much a result about amenable, abelian operator algebras, as opposed to amenable, abelian Banach algebras.
\smallskip

%
%===================================================================
%

The first positive result with respect to Question~A is due to M.V. \v Se\u\i nberg~\cite{Sei1977}:

\begin{theorem} \label{Seinberg-uniform} \emph{\textbf{[M.V. \v Se\u\i nberg]}}
If $\Omega$ is a compact Hausdorff space and $\cA \subseteq \cC(\Omega)$ is an amenable, uniform algebra that separates points, then $\cA = \cC(\Omega)$.
\end{theorem}

%
%===================================================================
%

For $T \in \cB(H)$, we denote by $\cA_T$ the norm-closed unital subalgebra of $\cB(H)$ generated by~$T$.

\begin{theorem} \label{Willis} \emph{\textbf{[G. Willis]}} ~\cite{Wil1995}
Let $K \in \cK(H)$.  If $\cA_K$ is amenable, then $K$ is similar to a diagonal operator.
\end{theorem}

The norm-closed algebra generated by a compact diagonal operator is self-adjoint.   As such, an immediate corollary to this Theorem is that if $K \in \cK(H)$ and $\cA_K$ is amenable, then $\cA_K$ is similar to a $C^*$-algebra.

%
%===================================================================
%

P.C.~Curtis and R.J.~Loy~\cite{CL1995} have proven that if $\cA \subseteq \cB(H)$ is amenable and generated by its normal elements, then $\cA = \cA^*$ is a $C^*$-algebra.

In~\cite{FFM2005, FFM2007}, D.~Farenick, B.E.~Forrest and the first author showed that if $T \in \cB(H)$ generates an amenable algebra $\cA_T$, and if $H$ admits an orthonormal basis $\{ e_n\}_{n=1}^\infty$ under which the matrix $[T] := [t_{ij}] = [ \langle T e_j, e_i \rangle]$ is upper triangular, then again, $T$ is similar to a normal operator $N$ with \textbf{Lavrientieff} spectrum.   That is, the spectrum $\sigma(T)$ of $T$ does not have interior, and it does not disconnect the complex plane.  As was shown by Lavrentieff~\cite{Lav1936}, this is precisely the property of the spectrum needed to ensure that the algebra of polynomials on $\sigma(T)$ is  dense in the space of continuous functions on $\sigma(T)$ with respect to the uniform norm, which implies that the algebra $\cA_N$ generated by $N$ is a $C^*$-algebra, and hence that $\cA_T$ is similar to $C^*(N)$. 

More recently, Y.~Choi~\cite{Cho2013} has shown (amongst other things) that  if $\cA$ is a closed, commutative  amenable subalgebra of a finite von Neumann algebra $\cM$, then $\cA$ must be similar to a $C^*$-algebra.     

%
%===================================================================
%

In a recent preprint of Y.~Choi, I.~Farah, and N.~Ozawa~\cite{CFO2013p},  Question A above has finally been resolved (in the negative).   There, the authors construct an ingenious example of a nonseparable and nonabelian amenable subalgebra of  $\ell_\infty(\mathbb{N}, \mathbb{M}_2(\mathbb{C}))$ which is not isomorphic to a nuclear $C^*$-algebra.  As they point out, their counterexample is ``inevitably nonseparable", and as we shall see, ``inevitably nonabelian".  The existence or nonexistence of a separable, amenable operator algebra which is not similar to a $C^*$-algebra remains an open problem.

%
%===================================================================
%

\bigskip

\subsection{} \label{TRP}
The current work is motivated by this problem in the case where the  algebra in question is \emph{abelian}.   Our main result is Theorem~\ref{MainTheorem}, which states that
\begin{quotation}
every \emph{abelian}, amenable operator algebra is similar to a (necessarily abelian, hence nuclear) $C^*$-algebra.
\end{quotation}
This result stands in stark contrast to the counterexample of Choi, Farah and Ozawa mentioned above.      
Our approach, however, takes us away from the notion of amenability proper, and is heavily influenced by the remarkable thesis of J.A.~Gifford~\cite{Gif1997} and his subsequent paper~\cite{Gif2006}.    

\smallskip

A particularly useful device in studying an operator algebra $\cA$ (i.e. a closed subalgebra of $\cB(H)$ for some Hilbert space $H$) is to examine its lattice of closed invariant subspaces, $\mathrm{Lat}\, \cA$.   It is elementary to see that the lattice $\mathrm{Lat}\, \cD$ of a $C^*$-algebra $\cD \subseteq \cB(H)$ has the property that if $M \in \mathrm{Lat}\, \cD$, then $M^\perp \in \mathrm{Lat}\, \cD$; in other words, every element of $\mathrm{Lat}\, \cD$ is orthogonally complemented.   We shall write $H = M \oplus M^\perp$ to denote the \emph{orthogonal} direct sum of the subspace $M$ and of $M^\perp$.    Given two closed subspaces $V$ and $W$ of $H$, we shall reserve the notation $H = V \dotplus W$ to mean that $V$ and $W$ are \textbf{topological complements} in $H$; that is, $H = V + W$, while $V \cap W = \{ 0\}$.

Suppose now that $\cD$ is a nuclear $C^*$-algebra, that $\varrho: \cD \to \cB(H_\varrho)$ is a continuous representation of $\cB$ and that $\cA := \overline{\varrho(\cD)}$.   By Christensen's Theorem~\cite{Chr1981}, there exists an invertible operator $S \in \cB(H_\varrho)$ so that $\tau := \mathrm{Ad}_S \circ \varrho$ is a ${}^*$-homomorphism.   From this it follows that the range of $\varrho$ is closed and that $\cB := \tau(\cD) = S^{-1} \cA S$ is a $C^*$-algebra.   A quick calculation shows that $\mathrm{Lat}\, \cA = S^{-1} \mathrm{Lat}\, \cB$.   

As such, given $M \in \mathrm{Lat}\, \cA$, we have that $S M \in \mathrm{Lat}\, \cB$, and thus $(S M)^\perp \in \mathrm{Lat}\, \cB$.   But then $H = S^{-1} H = S^{-1} ( (SM) \oplus (SM)^\perp ) = M \dotplus S^{-1} (S M)^\perp$ shows that $M$ is topologically complemented in $\mathrm{Lat}\, \cA$ by the element $S^{-1} (SM)^\perp$ of $\mathrm{Lat}\, \cA$.

We say that an operator algebra $\cA \subseteq \cB(H)$ has the \textbf{reduction property} if every element of its invariant subspace lattice $\mathrm{Lat}\, \cA$ is topologically complemented in $\mathrm{Lat}\, \cA$.    The above argument shows that if $\cA$ is the homomorphic image of a nuclear $C^*$-algebra, or more generally if $\cA$ is similar to a $C^*$-algebra, then $\cA$ has the reduction property.

That the lattice of invariant subspaces of an operator algebra being complemented reveals a great deal of structure about the algebra and its generators has been the theme of more than one paper.       For example, C.K.~Fong~\cite{Fon1977} closely examined the relationship between the reduction property of an operator algebra $\cA$ and the boundedness of certain graph transformations for $\cA$.   Later, S.~Rosenoer~\cite{Ros1987, Ros1993} showed amongst other things that if $T \in \cB(H)$ is an operator for which $\cA_T$ has the reduction property, and if $T$ commutes with an injective compact operator with dense range, then $T$ is similar to a normal operator.   Furthermore, he showed that every unital, strongly closed operator algebra $\cA$ with the reduction property and with the property that the ranges of the compact operators in $\cA$ span the underlying Hilbert space is reflexive:  that is, $\cA$ coincides with the algebra $\mathrm{Alg}\, \mathrm{Lat}\, \cA$ of all operators on $H$ which leave invariant each element of $\mathrm{Lat}\, \cA$.   (Both Fong's and Rosenoer's results are actually stated for operators on a Banach space - we shall not require those results here.)

In his thesis~\cite{Gif1997} (alternatively, see~\cite{Gif2006}), J.A.~Gifford defined a stronger version of the reduction property which he refers to as the \emph{total reduction property}:

%
%===================================================================
%

\begin{definition} \label{TRP2}
Let $\cA$ be a Banach algebra of operators acting on a Hilbert space $H$.   We say that $\cA$ has the \textbf{total reduction property (TRP)} if, for every continuous representation $\varrho: \cA \to \cB(H_\varrho)$ of $\cA$ as bounded linear operators on a Hilbert space $H_\varrho$, we have that the operator algebra $\overline{\varrho(\cA)}$ has the reduction property as a subalgebra of $\cB(H_\varrho)$.   
\end{definition}

Following~\cite{FFM2005}, we shall say that an operator $T$ has the \textbf{total reduction property} if $\cA_T$ does.

\bigskip

%
%===================================================================
%

Insofar as we are concerned, a particularly attractive relationship exists between the total reduction property and amenability:

\begin{theorem} \label{Gifford_totally_reductive} \emph{\textbf{[J.A.~Gifford]}}~\cite{Gif2006}
If $\cA \subseteq \cB(H)$ is an amenable Banach algebra of operators on a Hilbert space $H$, then $\cA$ has the total reduction property.
\end{theorem}

%
%===================================================================
%

\bigskip

Armed with this notion, Gifford obtained a far-reaching and beautiful generalization of Willis's result.

\begin{theorem} \label{Gifford} \emph{\textbf{[J.A.~Gifford]}}~\cite{Gif2006}
If $\cA \subseteq \cK(H)$ is a subalgebra of compact operators, then $\cA$ has the total reduction property if and only if  $\cA$ is similar to a $C^*$-algebra.    As a consequence, every amenable subalgebra of $\cK(H)$ is similar to a $C^*$-algebra.
\end{theorem}

\smallskip

In fact, Gifford proved this result under a slightly weaker hypothesis for $\cA$, namely that $\cA$ has the \emph{complete reduction property}, which is the statement that the algebra $\cA^{(\infty)} := \{ A \oplus A \oplus \cdots: A \in \cA \} \subseteq \cB(H^{(\infty)})$ has the reduction property.  

%
%===================================================================
%

 \bigskip

Suppose that an abelian algebra $\cA \subseteq \cB(H)$ is similar to a $C^*$-algebra $\cD$, say $\cA = S^{-1} \cD S$ for some invertible operator $S \in \cB(H)$.   Let $\varrho: \cA \to \cB(H_\varrho)$ be a (continuous) representation of $\cA$.  Then $\tau: \cD \to \cB(H_\varrho)$ defined by $\tau (D) = \varrho (S^{-1} D S)$ defines a continuous representation of $\cD$.   The argument of Section~\ref{TRP} above shows that the lattice $\mathrm{Lat}\, \overline{\tau(D)} = \mathrm{Lat}\, \overline{\varrho(A)}$ is topologically complemented, and thus  $\cA$ has the TRP.   

\bigskip

Our main result, Theorem~\ref{MainTheorem} establishes the converse:   if $\cA \subseteq \cB(H)$ is an abelian Banach algebra which has the TRP, then $\cA$ is similar to a $C^*$-algebra.    In particular, this confirms a conjecture of Gifford~\cite{Gif2006} in the abelian setting.

\vskip 0.5 cm

It is a pleasure for the authors to acknowledge the helpful conversations, insights and inspirations provided to us by Heydar Radjavi and Dilian Yang.   
%
%===================================================================
%
%=========================SECTION TWO  ++++++  THE MAIN RESULT=======
%
%===================================================================
%

\section{The main result.}

\subsection{} \label{intro-section2}
Our ultimate goal is to show that if an abelian operator algebra $\cA \subseteq \cB(H)$ has the total reduction property, and if $\Sigma_\cA$ denotes the maximal ideal space of $\cA$,  then the Gelfand Transform $\Gamma: \cA \to \cC(\Sigma_\cA)$ is a topological isomorphism.   This approach is motivated by the following.

%
%===================================================================
%

%\begin{theorem} \label{Seinberg1977}
%	\emph{\textbf{[M.V. \v{S}e\u{\i}nberg]}}~\cite{Sei1977} \ If $\Omega$ is a compact Hausdorff space and $\cA \subseteq C(\Omega)$ is an amenable uniform algebra that separates points, then $\cA = C(\Omega)$.
%\end{theorem}

%
%===================================================================
%

In his thesis, J.A.~Gifford provides the following analogue of \v{S}e\u{\i}nberg's Theorem~\ref{Seinberg-uniform} for  total reduction algebras (part (a) below).   As he mentions there, his proof owes much to the original.

\begin{theorem} \label{Gifford1997} \emph{\textbf{[J.A.~Gifford]}}~\cite{Gif1997}
Let $\cA \subseteq \bofh$ be an abelian, total reduction algebra.
\begin{enumerate}
	\item[(a)]
	If $\cA$ is contained in an abelian $C^*$-algebra $\cB \subseteq \bofh$, then $\cA$ is self-adjoint.
	\item[(b)]
	If $\cA$ is isomorphic to a closed subalgebra of an abelian $C^*$-algebra, then $\cA$ is similar to a $C^*$-algebra.
\end{enumerate}
\end{theorem}
		
%
%===================================================================
%

%We shall frequently allude to the following result of Gifford, and hence we include it here for the reader's convenience:
%
%\begin{proposition} \label{moduleprojection} \emph{\textbf{[J.A.~Gifford]}} \emph{(}\cite{Gif2005}, Lemma~1.5\emph{)}
%Let $\cA \subseteq \cB(H)$ be an operator algebra and $\cK$ be a left Hilbertian $\cA$-module with the complete reduction property.  Let $\theta: \cA \to \cB(\cK)$ be the associated representation.   There exists $M \geq 1$ so that for any submodule $V \subseteq \cK$, there is a module projection $E$ of $\cK$ onto $V$ which satisfies:
%\begin{enumerate}
%	\item[(a)]
%	$E$ lies in  the commutant $\theta(\cA)^\prime = \{ T \in \cB(\cK): T \theta(a) = \theta(a) T \mbox{ for all } a \in \cA \} $ of $\cA$, and 
%	\item[(b)]
%	$\norm E \norm \leq M$.
%\end{enumerate}	
%\end{proposition}
%

%
%===================================================================
%

The next result, again due to Gifford, shows that operator algebras $\cA$ with the total reduction property have a very rigid invariant subspace lattice \emph{under any continuous representation}.  Following the terminology in~\cite{Gif2006}, we refer to  idempotents in $\cB(H)$ as \emph{projections}, and we refer to self-adjoint projections as \emph{orthogonal projections}.

\bigskip

\begin{theorem} \label{projectionconstant} \emph{\textbf{[J.A.~Gifford]  Lemma~1.7}}~\cite{Gif2006}
Let $\cA$ be an operator algebra with the total reduction property.   Then there exists an increasing function $\kappa: \mathbb{R}^+ \to \mathbb{R}^+$ such that if $\theta: \cA \to \cB(H_\theta)$ is a continuous representation of $\cA$ and if $M \subseteq H_\theta$ is an invariant subspace for $\theta(\cA)$, then there exists a projection $E \in (\theta(\cA))^\prime = \{ T \in \cB(H_\theta): \theta(A) T = T \theta(A) \mbox{ for all } A \in \cA\}$ such that $\mathrm{ran}\, E = M$ and $\norm {E} \le \kappa(\norm{\theta})$.
\end{theorem}

\bigskip
\noindent{\textbf{Note:}}  For the sake of convenience below, we may and do assume that  $\kappa(t) > 1$ for all $t \ge 0$.
\smallskip

Upon fixing a representation $\theta: \cA \to \cB(H_\theta)$, the corresponding real number $\kappa(\norm{\theta})$ is referred to as the \textbf{projection constant} for the representation $\theta$ (or the projection constant for $\overline{\theta(\cA)}$).   Our strategy is to show that the projection constant  imposes a fixed bound on the norm of $T$ in terms of the norm of $T^2$ for all $T \in \cA$, which we then show to be precisely the result required to prove that the spectral radius on $\cA$ is a norm on $\cA$ which is equivalent to the operator norm.

%
%===================================================================
%

\subsection{}
The following proposition is motivated by results of Arveson~\cite{Arv1967}.  Recall that if $\cS \subseteq \cB(H)$ is a non-empty set, then $\cS^{(2)} = \{ S \oplus S : S \in \cS\} \subseteq \cB(H^{(2)}) = \cB(H \oplus H)$.   By a \textbf{linear manifold} in a Hilbert space $H$, we mean a vector subspace $L$ of $H$ which need not be closed in the norm topology on $H$.

\begin{proposition}\label{structure-IS}
Let $\cA \subseteq \cB(H)$ be an algebra with the  total reduction property.   Let $\kappa(\cdot)$ denote the projection function for $\cA$, and let $\kappa := \kappa(1)$. If $\cN \in \Lat\cA^{(2)}$, then there exist $Y\in\Lat\cA$, an $\cA$-invariant linear manifold $L \subseteq H$, and a closed linear map $R: L\to H$ satisfying $RTz=TRz$ for all $T\in\cA$ and $z\in L$ such that 
\[
\cN=(0\oplus Y)\dotplus \{(z,Rz)\mid z\in L\}. \]
Moreover, the projection $P_Y$ of $\cN$ onto $0\oplus Y$ along $\{(z,Rz)\mid z\in L\}$ has norm at most~$\kappa$.
\end{proposition}

\begin{proof}
Consider $\theta:\cA\to\cB(\cN)$ defined by $\theta(T)=(T\oplus T)|_{\cN}$.  Then $\theta$ is a representation of $\cA$ satisfying $\norm\theta\le 1$.
Let $Y=\{y\in\cH\mid (0,y)\in\cN\}$, so that $0 \oplus Y =\cN\cap(0\oplus H)$. Since $0 \oplus Y \in \Lat \theta(\cA)$, we have that $Y\in\Lat\cA$.   It follows from Theorem~\ref{projectionconstant} that there is a projection $P_Y=P_Y^2\in(\theta(\cA))'$   such that $P_Y\cN=(0\oplus Y)$ and $\norm{P_Y} \le \kappa$.

Let $\cN_0=\ker P_Y$, and observe that $\cN_0 \in \mathrm{Lat}\, \theta(\cA)$. Furthermore, $\cN= \mathrm{ran}\, P_Y \dotplus \ker\, P_Y = (0\oplus Y)\dotplus\cN_0$. Define 
\[
L=\{x\in H \mid (x,y)\in\cN_0\mbox{ for some }y\in H\}. \]
We claim that for each $x\in L$, there is a unique $y\in H$ such that $(x,y)\in\cN_0$. Indeed, if $y_1,y_2\in H$ are such that 
$$
(x,y_1) \quad\mbox{and}\quad(x,y_2)\in\cN_0,
$$
then 
$$
(x,y_1)-(x,y_2)=(0,y_1-y_2)\in\cN_0.
$$
However, from the definition of $Y$, we also have that $(0,y_1-y_2)\in(0\oplus Y)$. Since $(0\oplus Y)\cap\cN_0=\{0\}$, we find that $y_1=y_2$.

It follows that we can define a map $R: L\to H$ by letting $R \, x$ be equal to the unique $y\in H$ for which $(x,y)\in\cN_0$. It is routine to verify that $R$ is a linear map. By the definition of~$L$, we get
$$
\cN_0=\{(x,Rx)\mid x\in L\},
$$
and since $\cN_0$ is closed as a subspace of $\cN$, $R$ is closed as a linear map. 
Finally, if $x\in L$ and $T\in\cA$, it follows from the fact that $\cN_0$ is $\cA$-invariant that 
$$
(Tx,T Rx)\in\cN_0.
$$
Since $Tx \in L$ and $R T x$ is the unique element of $H$ so that $(Tx, R Tx) \in \cN_0$, we may conclude that $T Rx=R Tx$.
\end{proof}

%
%===================================================================
%

The next result provides the key estimate we shall require to prove our main theorem.

\begin{theorem}\label{quasinilpotent}
Let $\cA \subseteq \cB(H)$ be an abelian operator algebra with the total reduction property.   Then there exists $\mu > 0$ so that for all $S \in \cA$, 
\[
\norm {S}^2 \le \mu \, \norm {S^2}. \]
%Then $\cA$ is semisimple.
\end{theorem}
\begin{proof}
%Suppose that $\cA$ is not semisimple. Then $\cA$ admits a non-zero quasinilpotent operator. 

 As before, we denote Gifford's projection function by $\kappa(\cdot)$, and we let $\kappa := \kappa(1)$.   We shall argue by contradiction.   Suppose that the result is false.  
%Let $N\in\bbN$ be such that $N \ge 3 \kappa$. 
%It is easy to see that there exists a quasinilpotent operator $S\in\cA$ such that $\norm{S}>N$ but $\norm{S^2}\le 1$. 
Then for any constant $\gamma > 1$, we may find an element $S (= S_\gamma)  \in \cA$ such that $\norm{S}^2 > \gamma^2$, (i.e. $\norm{S} > \gamma$), while $\norm{S^2} \le 1$.  It will be convenient to first assume that $\gamma > 3 \kappa$.

Define 
$$
\cM=\{(h,Sh)\mid h\in H\}.
$$
Since $S$ is continuous, $\cM$ is a closed subspace of $H^{(2)}$, being the graph of $S$.
Since $\cA$ is abelian, $\cM\in\Lat\cA^{(2)}$. By the total reduction property, there exists a projection $P \in (\cA^{(2)})^\prime$ so that $P H^{(2)} = \cM$ and $\norm{P} \le \kappa$.   Let $\cN := \ker\, P \in \Lat\, \cA^{(2)}$.  Then $H^{(2)} = \cM \dotplus \cN$.   

%If we set $\cM$ is complemented by an $\cA^{(2)}$-invariant subspace~$\cN$, via a projection $P\in(\cA^{(2)})'$ of norm at most~$\kappa$. That is, $P H^{(2)}=\cM$, $\ker P=\cN$ and $\norm{P}\le \kappa$.

By Proposition~\ref{structure-IS}, $\cN$ decomposes into a topological direct sum of $\cA^{(2)}$-invariant subspaces as 
$$
\cN=(0\oplus Y)\dotplus \{(z,Rz)\mid z\in L\},
$$
where $Y$, $L$ and $R$ are as described in that Proposition. Moreover, the projection $P_Y$ of $\cN$ onto $0\oplus Y$ along $\{(z,Rz)\mid z\in L\}$ corresponding to this decomposition is of norm at most~$\kappa$. Thus $H^{(2)}$ decomposes into a topological direct sum of $\cA^{(2)}$-invariant subspaces as 
$$
H^{(2)}=\cM\dotplus(0\oplus Y)\dotplus \{(z,Rz)\mid z\in L\}.
$$
That is, we have: for each pair $(u,v)\in\cH^{(2)}$, there exist unique vectors $h\in\cH$, $y\in Y$ and $z\in L$ such that
\[
(u, v) = (h, Sh) + (0, y) + (z, Rz), \]
or equivalently,
\begin{equation}
\label{main-eqn}
\left\{\begin{array}{r}
h+z=u,\\
Sh+Rz+y=v.
\end{array}\right.
\end{equation}
Based on this equation, we obtain:
$$
P(u,v)=(h,Sh),
$$
so that 
$$
\norm{(h,Sh)}\le \kappa \, \norm{(u,v)}.
$$
Let $Q=I-P$, and note that $\norm Q\le  \kappa+1$. Clearly,
$$
Q(u,v)=(z,Rz+y),
$$
and thus 
$$
\norm{(z,Rz+y)}\le(\kappa+1)\, \norm{(u,v)}.
$$
Also,
$$
P_Y(z,Rz+y)=(0,y),
$$
and since $\norm {P_Y} \le \kappa$, we have
\begin{equation}
\label{eq:y}
\norm y\le (\kappa^2+\kappa)\cdot\norm{(u,v)}.
\end{equation}

%-------------------------
\bigskip
\noindent \emph{Claim 1}. There exists $z_L\in L$ such that $\norm{z_L}=1$ and $\norm{Sz_L}> \frac{\gamma}{3 \kappa}$.

Indeed, suppose that for all $z\in L$ we have $\norm{Sz}\le \frac{\gamma}{3 \kappa}\norm{z}$. Pick a vector $x_1\in H$ such that $\norm {x_1}=1$ and $\norm{Sx_1}> \gamma$. In  equation~\eqref{main-eqn}, let us use
$$
u=x_1\quad\mbox{and}\quad v=0.
$$
Then, in particular, $\norm{(u,v)} = \norm {(x_1, 0)}=1$.  Consider the unique decomposition 
\[
(x_1, 0) = (h_1, Sh_1) + (0,y_1) + (z_1, Rz_1).\]   Clearly,
$h_1=x_1-z_1\quad\mbox{and}\quad Sh_1=S x_1-Sz_1.$
\begin{itemize}
	\item{}
	Suppose first that $\norm{z_1}> \kappa+1$. 
	
	Then $\norm{h_1}\ge\norm{z_1}-\norm{x_1}> (\kappa+1)-1=\kappa$. Therefore $\norm{P}\ge\norm{P(x_1,0)}=\norm{(h_1,Sh_1)}\ge\norm{h_1}> \kappa$. This is a contradiction as $\norm P\le \kappa$. 
	\item{}
	Hence, $\norm{z_1}\le \kappa +1$. 
	
	But then
$$
\norm{Sh_1}\ge\norm{Sx_1}-\norm{Sz_1}> \gamma - \frac{\gamma}{3 \kappa}\norm{z_1} \ge \gamma - \frac{(\kappa + 1)\gamma}{3 \kappa} > \frac{\gamma}{3},
$$
since we have assumed that $\kappa > 1$.  Since we are also assuming that $\gamma > 3 \kappa$, it follows that $\norm P\ge \norm{Sh_1} > \kappa$, which is again a contradiction. 
\end{itemize}	
This proves \emph{Claim~1}.   It is worth noting that this shows that $L \not = \{ 0 \}$.
%-------------------------

\vspace{5mm}

Fix $z_L \in L$ satisfying the conditions of \emph{Claim~1}, namely:  $\norm {z_L} = 1$ and $\norm {S z_L} > \frac{\gamma}{3 \kappa}$.

Setting $u = z_L$ and $v = 0$, let us choose $h_0 \in H$, $y_0 \in Y$ and $z_0 \in L$ satisfying equation~(\ref{main-eqn}) above; that is, 
\[
(z_L, 0) = (h_0, S h_0) + (0, y_0) + (z_0, R z_0). \]

Our goal is to  show that 
\[
Sh_0 =  S z_L - (S -R)^{-1} (S^2 z_L + S y_0). \]

To see why this is useful, we shall first obtain explicit estimates to show that we can control  $\norm {S^2 z_L} $, $\norm {S y_0}$ and the norm of $(S-R)^{-1}|_{(S-R) L}$.   This will show that in terms of estimating the norm of $S h_0$, the dominant term in this decomposition of $S h_0$ is $S z_L$, whose norm we can choose sufficiently large (by selecting $\gamma$ sufficiently large) so as to force the norm of the associated projection $P$ to surpass the fixed bound coming from Gifford's projection constant $\kappa$, thereby producing the contradiction which completes our argument.  

Observe that by hypothesis, $\norm {S^2} \le 1$, and since $\norm {z_L} \le 1$, we have $\norm {S^2 z_L} \le 1$.  This term will not cause problems.  Moreover, since $(0, y_0) = P_Y \circ Q (z_L, 0)$, the argument which precedes \emph{Claim~1} shows that 
\[
\norm {y_0} \le (\kappa^2 + \kappa) \norm {(z_L, 0)} \le (\kappa^2 + \kappa).\]

\bigskip
\noindent
\emph{Claim 2}. For all nonzero $y\in Y$ we have $\norm{Sy}<2 \kappa \norm y$.

Suppose that this is not true.   Then there must exist an element $y_2\in Y$ with $\norm{y_2}=1$ and $\norm{Sy_2}\ge 2 \kappa$. Consider  equation~\eqref{main-eqn} with parameters $u=y_2$ and $v=0$ and observe that $(y_2, 0) = (y_2, S y_2) + (0, - S y_2) + (0, 0)$,  so that the triple $(h, y, z) = (y_2, -S y_2, 0)$  is a solution to this equation. (Note that $-S y_2$ belongs to $Y$ because $Y$ is $\cA$-invariant.)  From the uniqueness of the solution, we obtain
$$
P(y_2,0)=(y_2,Sy_2).
$$
It follows that $\norm{P}\ge\norm{P(y_2,0)}\ge\norm{Sy_2}\ge 2 \kappa$, a contradiction.
This proves \emph{Claim~2}.

\smallskip

When applied to the vector $y_0 \in Y$ above, we conclude that $\norm {S y_0} \le 2\kappa (\kappa^2 + \kappa)$.

%-------------------------

\bigskip
\noindent
\emph{Claim 3}. For every non-zero $z\in L$ we have $\norm{(S-R)z}>\frac{1}{2 \kappa}\norm{z}$. 

Suppose that the assertion of the claim is not true. Then there is a vector $z_3\in L$ such that $\norm{z_3}=2 \kappa$ and $\norm{(S-R)z_3}\le 1$. Consider equation~(\ref{main-eqn}) with the parameters 
$$
u=0\quad\mbox{and}\quad v=(S-R)z_3.
$$
Then $(0, (S-R)z_3) = (z_3, S z_3) + (0, 0) + (-z_3, -R z_3)$, and so
clearly, the triple $(h, y, z) := (z_3,0,-z_3)$ is a solution to equation~(\ref{main-eqn}). By the uniqueness of the solution,
$$
P(0,(S-R)z_3)=(z_3,Sz_3).
$$
Since $\norm{(S-R)z_3}\le 1$ and $\norm {z_3} = 2 \kappa$, we find that $\norm{P}\ge\norm{z_3} > \kappa$, which is a contradiction. This proves \emph{Claim~3}.

Note that in particular, \emph{Claim~3} implies that 
\begin{enumerate}
	\item[(i)]
	$(S-R)|_L$ is injective, and that 
	\item[(ii)]
	$(S-R)^{-1}:(S-R) L\to L$ is a bounded linear map and $\norm{(S-R)^{-1}} \le 2 \kappa$.
\end{enumerate}

%-------------------------

\vspace{5mm}

Returning to our goal:  from the equation $(z_L, 0) = (h_0, S h_0)  + (0, y_0) + (z_0, R z_0)$, we see that $h_0 = z_L - z_0 \in L$ and $Sh_0=-(Rz_0+y_0)$. Since $h_0\in L$, we may, in particular, apply $R$ to~$h_0$. We obtain
\begin{align*}
S z_L + y_0 
	&= (S-R) z_L + (R z_L + y_0) \\
	&= (S-R) z_L +  R (z_0 + h_0) + y_0 \\
	&= (S-R) z_L + R h_0 + (R z_0 + y_0) \\
	&= (S-R) z_L + (R - S) h_0.
\end{align*}	
Since $z_L, h_0 \in L$,  it follows that $w_0 := S z_L + y_0 \in (S-R) L$.

Now $w_0 \in (S-R)L$, and thus $S  w_0 \in S (S-R) L$.   But $S$ and $R$ commute when restricted to $L$, and so $S w_0 \in (S-R) S L$.   Since $L$ is $\cA$-invariant and $S \in \cA$, we have shown that $S w_0 \in  (S-R) L$.

Furthermore, $w_0, (S-R) z_L$ and $(R-S)h_0 \in (S-R) L$ implies that  $(S-R)^{-1} w_0 = z_L - h_0$.   Hence 
\begin{equation}
\label{eq:w_0}
S h_0 = S ( z_L - (S-R)^{-1} w_0) = S z_L - S (S-R)^{-1} w_0. 
\end{equation}

%-------------------------

\bigskip
%\vfill\newpage

\noindent
\emph{Claim 4}. $S(S-R)^{-1}w_0=(S-R)^{-1}Sw_0$.

Recalling that $S w_0 \in (S-R)L$, we have that $(S-R)^{-1} S w_0 \in L$ and
$$
(S-R)(S-R)^{-1}Sw_0 =Sw_0.
$$
Meanwhile,
$$
(S-R)S(S-R)^{-1}w_0 =S(S-R)(S-R)^{-1}w_0 =Sw_0,
$$
where the first identity  follows from the fact that $R$ commutes with $\cA$ on~$L$. But\linebreak ${(S-R)^{-1} w_0 \in L}$ and so $S (S-R)^{-1} w_0 \in L$ as $L$ is $\cA$-invariant.  Since $(S-R)|_L$ is injective as noted at the end of \emph{Claim~3}, this proves \emph{Claim~4}.
%-------------------------

\vspace{5mm}

We have demonstrated that $S w_0 = (S^2 z_L + S y_0) \in (S-R)L$ and so by equation~(\ref{eq:w_0}) and \emph{Claim~4}, 
\[
Sh_0 =  S z_L - (S -R)^{-1} (S^2 z_L + S y_0), \]
as we desired.

We will now use this to estimate the norm of $P$.

%-------------------------

\bigskip

\vspace{5mm}
Consider the following:
\begin{align*}
\norm{P}
	&\ge	\norm{P(z_L,0)} = \norm {(h_0, S h_0)}  \\
	&\ge \norm {S h_0} \\
	&= \norm{Sz_L-(S-R)^{-1} (S^2 z_L + S y_0)}\\
	&\ge \norm{Sz_L}-\norm{(S-R)^{-1}} \big(\norm{S^2 z_L} + \norm {S y_0} \big) \\
	&> \frac{\gamma}{3\kappa}  -2 \kappa (1+ 2 \kappa (\kappa^2+ \kappa)).\\
\end{align*}
By choosing $\gamma$  sufficiently large,  we find that the norm of the corresponding $P$ is  larger than~$ \kappa$, which is a contradiction.
\end{proof}

%
%===================================================================
%

\begin{remarks}
\begin{enumerate}
	\item[(a)]
	In fact, the proof shows that we may choose $\mu$ to be any constant greater than $(9 \kappa^2 + 12 \kappa^4 + 12 \kappa^5)^2$ in the statement of the above Theorem, where $\kappa = \kappa(1)$ is Gifford's projection constant for $\cA$.
	\item[(b)]
	A careful examination of the proof of Theorem~\ref{quasinilpotent} shows that the only place where we used the fact that the algebra $\cA$ is abelian was to conclude that the space $\cM := \{(h, Sh): h \in H\}$ is invariant for $\cA$.  For this, however, it is sufficient that $S$ lie in the centre $\cZ(\cA) := \{ Z \in \cA : Z A = A Z \mbox{ for all } A \in \cA \}$ of $\cA$.  Thus, even if $\cA$ is not abelian, so long as it has the total reduction property, the proof of Theorem~\ref{quasinilpotent} asserts the existence of a universal constant $\mu > 0$ so that if $S \in \cZ(\cA)$, then $\norm {S}^2 \le \mu \norm {S^2}$.
	
	Now suppose that $\cA$ is a non-abelian, amenable operator algebra and that $0 \not = T$ lies both in $\cZ(\cA)$ and in the Jacobson radical of $\cA$.   By virtue of the fact that $T$ is quasinilpotent, given $\varepsilon > 0$, there exists some $n \ge 1$ so that $ \norm {T^{2^{n+1}}} < \varepsilon \norm {T^{2n}}^{2}$.   But then with $S = T^{2^n} \in \cZ(\cA)$, we see that $\norm {S^2} < \varepsilon \norm {S}^2$.   Since $\varepsilon  > 0$ is arbitrary, this leads to a contradiction.
	
	The conclusion is that if $\cA$ is an amenable operator algebra, then the intersection of the centre of $\cA$ with the radical of $\cA$ is $\{ 0 \}$.  In the case where $\cA$ is abelian, this is the statement that $\cA$ is semisimple.   But as we shall now see, in the abelian case, much more is true.
\end{enumerate}	
\end{remarks}

%
%===================================================================
%

The next  Proposition is standard.  We include the proof for the convenience of the reader.

\begin{proposition} \label{MUAHAHA}
Let $(\cA, \norm{\cdot})$ be an abelian Banach algebra and suppose that there exists a constant $\mu > 0$ such that 
\[
\norm {x}^2 \le \mu \ \norm {x^2}  \mbox{\ \ \ \ for all } x \in \cA. \]

Then the spectral radius function $\mathrm{spr}\,(\cdot)$ is a norm on $\cA$ which is equivalent to the given norm $\norm {\cdot}$.
\end{proposition}

\begin{proof}
It is well-known that $\mathrm{spr}\, (\cdot)$ is a seminorm on $\cA$.
Fix $x \in \cA$.    Without loss of generality, $\mu \ge 1$.

\smallskip
It is clear that $\mathrm{spr}\, (x) \le \norm {x} $.

\smallskip
Conversely, for any $x \in \cA$, $\norm {x}^2 \le \mu \ \norm {x^2} $ implies that for each $n \ge 1$, 
\begin{align*}
\norm x  \
	&\le \mu^\frac{1}{2} \norm {x^2}^\frac{1}{2} \\
	&\le \mu^\frac{1}{2} \big( \mu^\frac{1}{2} \norm {x^4}^\frac{1}{2} \big)^{\frac{1}{2}} \\
	&= \mu^{\frac{1}{2} + \frac{1}{4}}  \norm {x^4}  \big)^{\frac{1}{4}} \\
	&\le \mu^{\frac{1}{2} + \frac{1}{4}} \big( \mu^{\frac{1}{2}} \norm {x^8}^{\frac{1}{2}} \big)^{\frac{1}{4}} \\
	&= \mu^{\frac{1}{2} + \frac{1}{4} + \frac{1}{8}}  \  \norm {x^8}^{\frac{1}{8}} \\
	&\le \cdots \\
	&\le \mu^{\frac{1}{2} + \frac{1}{4} + \cdots + \frac{1}{2^n}}  \ \big( \norm {x^{2^n}}^{\frac{1}{2^n}} \big)\\
\end{align*}	

Taking limits as $n$ tends to infinity shows that 
\[
\norm {x} \le \mu \ \mathrm{spr}\, (x). \]

Since $\mu$ was independent of $x$, 
\[
\mathrm{spr}\, (x) \le \norm {x}  \le \mu \ \mathrm{spr}\, (x) \mbox{ \ \ \ \ \ for all } x \in \cA. \]
This completes the proof.
\end{proof}

%\bigskip

%
%===================================================================
%

\subsection{}
Let $\cA \subseteq \cB(H)$ be an abelian algebra with the total reduction property.  Recall that $\Gamma: \cA \to \cC(\Sigma_\cA)$ denotes the Gelfand Transform of $\cA$ into the space of continuous functions on the maximal ideal space $\Sigma_\cA$ of $\cA$ and that $\mathrm{spr}(x) = \norm {\Gamma(x)}$ for all $x \in \cA$.   

\bigskip

%
%===================================================================
%

We are now in a position to prove our Main Theorem.

\begin{theorem} \label{MainTheorem}
Let $H$ be a complex Hilbert space and  $\cA$ be a closed, abelian subalgebra of $\cB(H)$.   The following conditions are equivalent:
\begin{enumerate}
	\item[(a)]
	$\cA$ is amenable;
	\item[(b)]
	$\cA$ has the total reduction property;
	\item[(c)]
	$\cA$ is similar to a $C^*$-algebra.
\end{enumerate}
\end{theorem}

\begin{proof}
\begin{enumerate}
	\item[(a)] implies (b):
	This is Theorem~\ref{Gifford_totally_reductive} above, due to Gifford.
	\item[(b)] implies (c):
	By Theorem~\ref{quasinilpotent}, there exists $\mu > 0$ so that $\norm {x}^2 \le \mu \norm {x^2}$ for all $x \in \cA$.   By Proposition~\ref{MUAHAHA}, the spectral radius is a norm on $\cA$ which is equivalent to the operator norm on $\cA$.
	
	As mentioned above, it follows that the Gelfand Transform $\Gamma: \cA \to \cC(\Sigma_\cA)$ is not only injective, but the range of $\Gamma$ is closed.   That is, $\cA$ is topologically isomorphic to the closed subalgebra $\Gamma(\cA)$ of $\cC(\Sigma_\cA)$.  Since $\cA$ has the total reduction property, so does $\Gamma(\cA)$, and we can now apply Theorem~\ref{Gifford1997} to conclude that $\cA$ is similar to a $C^*$-algebra.
	\item[(c)] implies (a):
	Since $\cA$ is abelian, if $\cA$ is similar to a $C^*$-algebra $\cB$, then $\cB$ must be abelian as well.   Thus $\cB$ is nuclear~\cite{Tak1964}, and therefore amenable~\cite{Haa1983}, and so $\cA$ is amenable, being similar to, and hence a homomorphic image of, an amenable algebra.
\end{enumerate}
\end{proof}

%
%===================================================================
%

\begin{corollary} \label{TRPoperators}
Let $H$ be a complex Hilbert space, and let $T \in \cB(H)$.   The following conditions are equivalent:
\begin{enumerate}
	\item[(a)]
	$\cA_T$ is amenable.
	\item[(b)]
	$\cA_T$ has the total reduction property.
	\item[(c)]
	$T$ is similar to a normal operator and the spectrum of $T$ is a Lavrentieff set.
\end{enumerate}	
\end{corollary}

\begin{proof}
\begin{enumerate}
	\item[(a)] implies (b):  
	As before, this is Theorem~\ref{Gifford_totally_reductive}.
	\item[(b)] implies (c):
	Since $\cA_T$ is clearly abelian, Theorem~\ref{MainTheorem} implies that  $\cA_T$ is similar to a $C^*$-algebra $\cB$, say 
	\[
	\cA_T = S^{-1} \cB S. \]
	But then $\cB = S \cA_T S^{-1} = \cA_{S T S^{-1}}.$  Since $\cB$ is selfadjoint and abelian, $N: = {S T S^{-1}}$ is normal.  That the spectrum of $T$ is a Lavrentieff set is Proposition~3.6 of~\cite{Mar2008}.
	\item[(c)] implies (a):
	Suppose that $T = S^{-1} N S$, where $S \in \cB(H)$ is invertible and $N$ is normal.   Since $\sigma(T) = \sigma(N)$ is a Lavrentieff set, $\cA_N = C^*(N)$(~\cite{FFM2005}, Theorem~2.7).   But then $\cA_T = S^{-1} \cA_N S = S^{-1} C^*(N) S$ is similar to an abelian and hence nuclear $C^*$-algebra, so that $\cA_T$ is amenable.
\end{enumerate}
\end{proof}

%
%===================================================================
%
%
%===================================================================
%
%
%===================================================================
%

\section{Consequences of the Main Theorem}

\subsection{}
The article~\cite{FFM2005} contained a number of results about singly generated, amenable operator algebras which relied upon the equivalence of conditions (a) and (c) of Corollary~\ref{TRPoperators} above.   Unfortunately, although that paper claimed a proof of this equivalence,  an error was later discovered (see~\cite{FFM2007}), and as a consequence, the results of Section~5 of~\cite{FFM2005} had to be withdrawn as well.   Now that the validity of Corollary~\ref{TRPoperators} has been established, we are able to retrieve some of those results, and to extend them beyond the singly generated case.  This having been said, the proofs here are often very similar to the original proofs.

%
%===================================================================
%

The following result provides a partial answer to a question of G.~Pisier~\cite{Pis2001}, p.~13.

\begin{corollary} \label{cor3.1} 
Let $\cA$ be a unital, abelian, amenable algebra.  If $\varphi: \cA \to \cB(H)$ is a bounded, unital homomorphism, then there exists  a contractive homomorphism $\rho: \cA \to \mathcal{B}({H})$ and an invertible operator $S \in \cB(H)$ such that $\varphi(x) = \mathrm{Ad}_S \circ \rho(x) = S^{-1} \rho(x) S$ for all $x \in \cA$.
\end{corollary}

\begin{proof}
Let $\cB = \overline{\varphi(\cA)}$.   Then $\cB$ is an abelian, amenable subalgebra of $\cB(H)$, and so by Theorem~\ref{MainTheorem}, $\cB$ is similar to an abelian $C^*$-algebra $\cC$, say $\cB = S^{-1} \cC S$ for some invertible operator $S \in \cB(H)$.  

Consider $\rho:\cA \to \cC$ defined by $\rho (x) = S \varphi(x) S^{-1}$.  Then $\rho$ is clearly a bounded homomorphism, and for each $x \in \cA$, $\rho(x) \in \cC$ implies that $\norm {\rho(x)} = \mathrm{spr}{(\rho(x))} \le \mathrm{spr}(x) \le \norm{x}$.  
\end{proof}

%
%===================================================================
%

\begin{corollary} \label{cor3.2}
Let $\cA$ be an abelian, amenable Banach algebra, and suppose that $\rho: \cA \to \cB(H)$ is a continuous representation of $\cA$.   Then $\rho(q)= 0$ for all $q \in \mathrm{Rad}\, (\cA)$.
\end{corollary}

\begin{proof}
If $\cB = \overline{\rho(\cA)}$, then $\cB$ is an abelian, amenable operator algebra, and by Theorem~\ref{MainTheorem}, $\cB$ is semisimple.   Since $\sigma(\rho(q)) \subseteq \sigma(q) = \{ 0\}$ for each $q \in \mathrm{Rad}\, (\cA)$, it follows that $\rho(q) = 0$.
\end{proof}

%
%===================================================================
%

\begin{corollary}  \label{cor3.3}
Suppose that $\cA \subseteq \cB(H)$ is a unital, abelian and amenable subalgebra.  Then $\cA + \cK(H)$ is norm-closed and amenable.
\end{corollary}

\begin{proof}
The proof of this result is an easy adaptation of that of Proposition~5.9 of~\cite{FFM2005}.  If $S \in \cB(H)$ is an invertible operator which implements the similarity between $\cA$ and an abelian $C^*$-algebra $\cC$ (the existence of which is guaranteed by Theorem~\ref{MainTheorem}), then $S$ implements the similarity between $\cA + \cK(H)$ and $\cC + \cK(H)$.   Since the latter is well-known to be a $C^*$-algebra, $\cA+\cK(H)$ is complete, hence closed.

Since any extension of a nuclear $C^*$-algebra -- $\cC$ --  by a nuclear algebra -- $\cK(H)$ -- is nuclear(\cite{BO2008} p.~105), it follows that $\cC + \cK(H)$ is nuclear, hence amenable.   But then $\cA + \cK(H)$ is amenable, being isomorphic to $\cC + \cK(H)$.
\end{proof}  

%
%===================================================================
%

\begin{proposition} \label{prop3.4}
Let $\cA \subseteq \cB(H)$ be an abelian, amenable operator algebra.   Let\linebreak $\rho: \cA \to \cB(H_\rho)$ be a continuous representation of $\cA$.   Then $\rho$ is completely bounded.  
\end{proposition}

\begin{proof}
By Theorem~\ref{MainTheorem}, there exists an invertible operator $T \in \cB(H)$ so that $\cB := \mathrm{Ad}_T (\cA) = T^{-1} \cA T$ is an abelian $C^*$-algebra.   As such, $\rho \circ \mathrm{Ad}_{T^{-1}}$ is a continuous representation of $\cB$.   But every continuous representation of an abelian $C^*$-algebra is similar to a ${}^*$-representation by Christensen's Theorem~\cite{Chr1981}, and so we can find $R \in \cB(H_\rho)$ so that $\tau :=\mathrm{Ad}_R \circ \rho \circ \mathrm{Ad}_{T^{-1}}$ is a ${}^*$-representation of $\cB$, and as such is completely contractive.  But then $\rho = \mathrm{Ad}_{R^{-1}} \circ \tau \circ \mathrm{Ad}_{T}$, so 
\begin{align*}
	\norm{\rho}_{cb} 
		&\le \norm {\mathrm{Ad}_{R^{-1}}}_{cb} \ \norm {\tau}_{cb} \ \norm {\mathrm{Ad}_{T}}_{cb} \\
		&\le \norm{R} \ \norm {R^{-1}} \ \norm {T} \ \norm{T^{-1}}.
\end{align*}
\end{proof}

%
%===================================================================
%

\begin{remark} \label{similaritydegree}
Recently, in studying the Kadison Similarity Problem, G.~Pisier has developed a rich and deep theory of ``length" and ``similarity degree" for operator algebras (see, for example,~\cite{Pis1999},~\cite{Pis2001},~\cite{Pis2001b}).  More precisely, if $\cA$ is a unital operator algebra, he defines the  \emph{length} $\ell(\cA)$ of $\cA$ to be the smallest positive integer $d$ for which there is a constant $K > 0$ such that for any $n \geq 1$ and any $x \in \bbM_n(\cA)$, there exists a positive integer $N = N(n, x)$ and a factorization
\[
x = \alpha_0 E_1 \alpha_1 E_2 \cdots E_d \alpha_d \]
where $\alpha_0 \in \bbM_{n, N}(\bbC), \alpha_d \in \bbM_{N, n}(\bbC), \alpha_j \in \bbM_N(\bbC)$ for $1 \leq j \leq {d-1}$, and where $E_j \in \bbM_N(\cA)$ are diagonal matrices satisfying
\[
\left( \prod_{j=0}^d \norm {\alpha_j} \right) \ \left( \prod_{k=1}^d \norm {E_k}  \right) \leq K \norm {x}. \]
In~\cite{Pis1999}, he defines the \emph{similarity degree} $d(\cA)$ to be the infimum over all $\beta \geq 0$ for which there exists $K > 0$ satisfying $ \norm {\varphi}_{cb} \leq K \norm {\varphi}^\beta $ whenever $\varphi$ is a unital homomorphism from $\cA$ into some $\cB(H)$, and proves that $d (\cA) = \ell(\cA)$.  (Here $\norm {\varphi}_{cb}$ denotes the \emph{completely bounded norm} of $\varphi$ -- see~\cite{Pis2001} or~\cite{Pau2002} for an introduction to  completely bounded maps and their properties.)  He also shows that the Kadison Similarity Problem admits a positive answer for all unital $C^*$-algebras $\cD$ if and only if there exists $d_0$ so that $\ell(\cD) \leq d_0$ for all $C^*$-algebras $\cD$.  

It is not very difficult to verify that if an operator algebra $\cB$ is similar to an operator algebra $\cA$, then $\ell(\cB) = \ell(\cA)$, and hence their similarity degrees also coincide.   By a result of J.~Bunce and E.~Christensen~\cite{Chr1982}, if $\cB$ is an abelian $C^*$-algebra, then either $\cB$ is finite dimensional, in which case $d(\cB) = 1$, or $\cB$ is infinite dimensional, and then $d(\cB) = 2$.   A simple consequence of Theorem~\ref{MainTheorem}, therefore, is that if $\cA$ is an abelian, amenable (infinite dimensional) operator algebra, then $\cA$ is similar to an abelian $C^*$-algebra $\cB$, whence $d(\cA) = \ell(\cA) = \ell(\cB) = 2$.

\end{remark}

%
%===================================================================
%
%\bibliographystyle{plain}
%\bibliography{papers}

%\begin{thebibliography}{00}
%\end{thebibliography}
%
%===================================================================
%

\end{document}